\documentclass{article}%
\usepackage{amsfonts, amsmath, amssymb}
\usepackage{subfig}
\usepackage{pstricks, pst-plot, pst-node, multido}
\usepackage{amsmath}
\usepackage{amsfonts}
\usepackage{amssymb}
\usepackage{graphicx}

\newtheorem{theorem}{Theorem}

\newtheorem{definition}[theorem]{Definition}

\newtheorem{proposition}[theorem]{Proposition}
\newtheorem{remark}[theorem]{Remark}

\newenvironment{proof}[1][Proof]{\noindent\textbf{#1.} }{\ \rule{0.5em}{0.5em}}
\begin{document}

\title{Bifurcation of periodic solutions from a ring configuration of discrete
nonlinear oscillators}
\author{
C. Garc\'{\i}a-Azpeitia
\and J. Ize
\\{\small Depto. Matem\'{a}ticas y Mec\'{a}nica, IIMAS-UNAM, FENOMEC, }
\\{\small Apdo. Postal 20-726, 01000 M\'{e}xico D.F. }
\\{\small cgazpe@hotmail.com}
\\{\small }
}

\maketitle

\begin{abstract}

This paper gives an analysis of the periodic solutions of a ring of $n$
oscillators coupled to their neighbors. We prove the bifurcation of branches
of such solutions from a relative equilibrium, and we study their
symmetries. We give complete results for a cubic Schr\"odinger potential and
for a saturable potential and for intervals of the amplitude of the
equilibrium. The tools for the analysis are the orthogonal degree and
representation of groups. The bifurcation of relative equilibria was given
in a previous paper.

\end{abstract}

\section{Introduction}

Consider a lattice of $n$ nonlinear oscillators, with a periodicity
condition or a ring of $n$ oscillators. The lattice may come from the
discretization of a nonlinear PDE, such as a nonlinear Schr\"odinger
equation, with different types of potentials. These equations have been used
very often as models of several phenomena of physics of waves, in particular
in nonlinear optics, see \cite{EJ03}.

In previous papers, we have studied the problems of bifurcation of relative
equilibria (see \cite{GaIz11}) and bifurcation of periodic solutions for
rings of masses (\cite{GaIz112}), vortices and filaments (\cite{GaIz113}).

Although these problems are different as models of physical phenomena, they
have in common several mathematical aspects, such as a similar linearization
and the same type of symmetries.

A relative equilibrium is a stationary solution of the system in a constant rotating
frame. In particular, we consider one where all oscillators have the same amplitude
but different phases. When these equilibria are localized they are called
breathers. In case of a periodic localized solution, one has a
quasi-periodic breather, see \cite{EJ03} or \cite{Jo04}.

Numerical studies of the bifurcation of relative equilibria for a circular
lattice and a defect in one of the oscillators, for $n=7$, is given in \cite%
{PD09}. For the case of a Dirichlet condition, instead of the periodic
problem, one has a very complete numerical analysis of bifurcation of
breathers in \cite{Pan10}.

The linearization of the system at a critical point is a $2{n}\times2{n}$
matrix, which is non invertible, due to the rotational symmetry. These facts
imply that the study of the spectrum of the linearization is not an easy
task and that the classical bifurcation results for periodic solutions may
not be applied directly. However, we shall use the change of variables
proved in our previous paper, \cite{GaIz11}, in order to give not only this
spectrum but also the consequences for the symmetries of the solutions.

The present paper is a continuation of \cite{GaIz11}. Thus, we shall use the
results in that paper, but we shall recall all the important notions. In our
parallel papers, \cite{GaIz112} and \cite{GaIz113}, we study similar
problems for point masses and vortices. Although there are many
similarities, in particular in the change of variables, the results are of a
quite different nature.

The next section is devoted to the mathematical setting of the problem, with
the symmetries involved. Then, we give, in the following two sections, the
preliminary results needed in order to apply the orthogonal degree theory
developed in \cite{IzVi03}, that is the global Liapunov-Schmidt reduction,
the study of the irreducible representations, with the change of variables
of \cite{GaIz11}, and the symmetries associated to these representations. In
the next section, we prove our bifurcation results and, in the following
section, we give the analysis of the spectrum, with the complete results on
the type of solutions which bifurcate from the relative equilibrium.

\section{Setting the problem}

Let us denote by $q_{j}(t)\in \mathbb{C}$ the $j$'th oscillator, for $j\in
\{1,...,n\}$. The dNLS equations are
\begin{equation*}
i\dot{q}_{j}=h(\left\Vert q_{j}\right\Vert
^{2})q_{j}+(q_{j+1}-2q_{j}+q_{j-1})\text{,}
\end{equation*}%
where $h$ is the nonlinear potential. We wish to study a finite circular
lattice, that is, a lattice of oscillators for $j\in \{1,...,n\}$, with
periodic conditions $q_{j}=q_{j+n}$.

The solutions of the form $q_{j}=e^{\omega ti}u_{j}$, with $u_{j}$ constant,
are called relative equilibria. In order to obtain the amplitude as a
parameter, we need to change coordinates, with $q_{j}=\mu e^{\omega ti}u_{j}$%
. In this manner, we have that the values $u_{j}$ form a relative
equilibrium when%
\begin{equation*}
-\omega u_{j}=h(\left\vert \mu u_{j}\right\vert
^{2})u_{j}+(u_{j+1}-2u_{j}+u_{j-1})\text{.}
\end{equation*}

\begin{remark}
Given that the lattice is integrable for $n=1$ and $n=2$, we shall look for
bifurcation for $n\geq3$.
\end{remark}

The starting point is a relative equilibrium which looks like a rotating
wave. We give a condition which needs to be satisfied by the potential for
the existence of this rotating wave.

\begin{proposition}
Define $a_{j}=e^{ij\zeta}$, with $\zeta=2\pi/n$, then $\bar{a}%
=(a_{1},...,a_{n})$ is a relative equilibrium if%
\begin{equation*}
\omega=4\sin^{2}(\zeta/2)-h(\mu^{2})\text{.}
\end{equation*}
\end{proposition}

The proof is simple and given in \cite{GaIz11}.

For the study of periodic solutions, we need to change to rotating
coordinates: $q_{j}=\mu e^{\omega ti}u_{j}$, and the equation becomes%
\begin{equation*}
i\dot{u}_{j}-\omega u_{j}=h(\left\vert \mu u_{j}\right\vert
^{2})u_{j}+(u_{j+1}-2u_{j}+u_{j-1}).
\end{equation*}
Changing to real coordinates, one has, for $u_{j}\in\mathbb{R}^{2}$%
\begin{equation*}
J\dot{u}_{j}=\omega u_{j}+h(\left\vert \mu u_{j}\right\vert
^{2})u_{j}+(u_{j+1}-2u_{j}+u_{j-1})\text{,}
\end{equation*}
where $J$ is the standard symplectic matrix.

Set $u=(u_{1},...,u_{n})^{T}$ be the vector of positions and let $\mathcal{J}%
=\emph{diag}(J,...,J)$, then the vectorial form of the equation is%
\begin{align*}
\mathcal{J}\dot{u} & =\nabla V(u)\text{ with} \\
V & =\frac{1}{2}\sum_{j=1}^{n}\left\{ H(u_{j})-\left\vert
u_{j+1}-u_{j}\right\vert ^{2}\right\} ,
\end{align*}
where $H(u)$ is the function which satisfies $\nabla H(u)=\omega
u+h(\left\vert \mu u\right\vert ^{2})u$.

Since we are looking for periodic solutions, let us define $x(t)=u(t/\nu )$.
Then, $2\pi /\nu $ solutions for $u$ are $2\pi $ solutions for $x$ of the
equation%
\begin{equation*}
f(x)=-\nu \mathcal{J}\dot{x}+\nabla V(x)=0\text{.}
\end{equation*}%
The operator $f$ is defined from $H_{2\pi }^{1}(\mathbb{R}^{2n})$ into $%
L_{2\pi }^{2}(\mathbb{R}^{2n})$.

\begin{definition}
Let $S_{n}$ be the group of permutations of $\{1,...,n\}$. One defines the
action of $S_{n}$ in $\mathbb{R}^{2n}$ as%
\begin{equation*}
\rho(\gamma)(x_{1},...,x_{n})=(x_{\gamma(1)},...,x_{\gamma(n)}),
\end{equation*}
and the action of  $\theta\in SO(2)$ as
\begin{equation*}
\rho(\theta)x=e^{-\mathcal{J}\theta}x\text{.}%
\end{equation*}

\end{definition}

Let $\mathbb{Z}_{n}$ be the subgroup of permutations
generated by $\zeta(j)=j+1$ modulus $n$. The gradient
$\nabla V$ is $\mathbb{Z}_{n}\times SO(2)$-equivariant, that is it commutes with the action of the group, and
 the map $f$ is $\Gamma\times S^{1}$-equivariant with the abelian group%
\begin{equation*}
\Gamma=\mathbb{Z}_{n}\times SO(2),
\end{equation*}
where the action of $S^{1}$ is by time translation.

Now, the infinitesimal generators of $S^{1}$ and $\Gamma$ are%
\begin{equation*}
\text{ }Ax=\frac{\partial}{\partial\varphi}|_{\varphi=0}x(t+\varphi)=\dot {x}%
\text{ and }A_{1}x=\frac{\partial}{\partial\theta}|_{\theta=0}e^{-\mathcal{%
J\theta}}x=-\mathcal{J}x\text{.}
\end{equation*}
Since $V$ is $\Gamma$-invariant, then the gradient $\nabla V(x)$ must be
orthogonal to the generator $A_{1}x$. As a consequence, the map $f$ must be $%
\Gamma\times S^{1}$-orthogonal, due to the equalities

\begin{align*}
\left\langle f(x),\dot{x}\right\rangle & =-\nu\left\langle \mathcal{J}\dot{x}%
,\dot{x}\right\rangle +V(x)|_{0}^{2\pi}=0 \\
\left\langle f(x),-\mathcal{J}x\right\rangle & =\nu\frac{1}{2}\left\langle
x,x\right\rangle |_{0}^{2\pi}-\left\langle \mathcal{J}x,\nabla
V(x)\right\rangle =0\text{.}
\end{align*}

Define $\mathbb{\tilde{Z}}_{n}$ as the subgroup of $\Gamma$ generated by $%
(\zeta,\zeta)\in\mathbb{Z}_{n}\times SO(2)$ with $\zeta=2\pi/n\in SO(2)$.
Since the action of $(\zeta,\zeta)$ leaves fixed the equilibrium $\bar{a}$,
then the isotropy group of $\bar{a}$ is the group ${\Gamma}_{\bar{a}}\times
S^{1}$ with%
\begin{equation*}
\Gamma_{\bar{a}}=\mathbb{\tilde{Z}}_{n}.
\end{equation*}
Thus, the orbit of $\bar{a}$ is isomorphic to the group $SO(2)$. In fact,
the orbit consists of the rotations of the equilibrium. As a consequence,
the generator of the orbit $A_{1}\bar{a}=-\mathcal{J}\bar{a}$ must be in the
kernel of $D^{2}f(\bar{a})$.

\section{The Liapunov-Schmidt reduction}

In order to apply the orthogonal degree of \cite{IzVi03}, one needs to make
a reduction of the bifurcation map to some finite space.

The bifurcation map $f$ has Fourier series%
\begin{equation*}
f(x)=\sum_{l\in \mathbb{Z}}(-l\nu i\mathcal{J}x_{l}+g_{l})e^{ilt}\text{,}
\end{equation*}%
where $x_{l}$ and $g_{l}$ are the Fourier modes of $x$ and $\nabla V(x)$.
Since $-il\nu (i\mathcal{J)}$ is invertible for all big $l\nu$'s, then one may
solve $x_{l}$ for $\left\vert l\right\vert >p$ and $\nu$ bounded from below,
 whenever $x(t)$ belongs to a
bounded set $\Omega $ in the space $H^{1}$, that is also bounded uniformly
in $\mathbb{R}^{2}$.

In this way, the bifurcation operator $f$ has the same zeros as the
bifurcation function
\begin{equation*}
f(x_{1},x_{2}(x_{1},\nu),\nu)=\sum_{\left\vert l\right\vert \leq p}(-l\nu i%
\mathcal{J}x_{l}+g_{l})e^{ilt}\text{,}
\end{equation*}
and the linearization of the bifurcation function at some equilibrium $\bar {%
a}$ is%
\begin{equation*}
f^{\prime}(\bar{a})x_{1}=\sum_{\left\vert l\right\vert \leq p}\left( - l\nu i%
\mathcal{J}+D^{2}V(\bar{a})\right) x_{l}e^{ilt}\text{.}
\end{equation*}
Here $x_{1}$ corresponds to the $2p+1$ first Fourier modes and $%
x_{2}(x_{1},\nu)$ is the result of applying the global implicit function
theorem for functions in $\Omega$ and $\nu$ bounded from below.

The linearization of the bifurcation function is determined by blocks $%
M(l\nu)$ for $l\in\{0,...,p\}$, where $M(\nu)$ is the matrix

\begin{equation*}
M(\nu)=-\nu i\mathcal{J}+D^{2}V(\bar{a})\text{.}
\end{equation*}
These blocks $M(l\nu)$ represent the Fourier modes of the linearized
equation at the equilibrium.

\section{Irreducible representations}

In order to apply the orthogonal degree, one needs to find the irreducible
representation subspaces for the action of $\Gamma_{\bar{a}}=\mathbb{\tilde {%
Z}}_{n}$.

For $k\in\{1,...,n\}$, we define the isomorphisms $T_{k}:\mathbb{C}%
^{2}\rightarrow W_{k}$ as%
\begin{align*}
T_{k}(w) & =(n^{-1/2}e^{(ikI+J)\zeta}w,...,n^{-1/2}e^{n(ikI+J)\zeta}w)\text{
with} \\
W_{k} & =\{(e^{(ikI+J)\zeta}w,...,e^{n(ikI+J)\zeta}w):w\in\mathbb{C}^{2}\}%
\text{.}
\end{align*}

In the paper \cite{GaIz11}, we have proved that the subspaces $W_{k}$ are
irreducible representations of the group $\mathbb{\tilde{Z}}_{n}$.
 Also, we showed that the action of $(\zeta,\zeta)\in\mathbb{\tilde{Z}}_{n}$
on the space $W_{k}$ is given by
\begin{equation*}
\rho(\zeta,\zeta)=e^{ik\zeta}\text{.}
\end{equation*}

Since the subspaces $W_{k}$ are orthogonal, then the linear map%
\begin{equation*}
Pw=\sum_{j=1}^{n}T(w_{k})
\end{equation*}
is orthogonal, where $w=(w_{1},...,w_{n})$.

Since the map $P$ rearranges the coordinates of the irreducible
representations, one has, from Schur's lemma, that%
\begin{equation*}
P^{-1}D^{2}V(\bar{a})P=diag(B_{1},...,B_{n})\text{,}
\end{equation*}
where $B_{k}$ are matrices which satisfy $D^{2}V(\bar{a}%
)T_{k}(w)=T_{k}(B_{k}w)$. In the paper \cite{GaIz11}, we have found the
blocks $B_{k}$: they satisfy $B_{n-k}=\bar{B}_{k}$ and we had the following
result:

Define $\alpha_{k}$\ and $\gamma_{k}$\ as%
\begin{equation*}
\alpha_{k}=4\cos\zeta\sin^{2}k\zeta/2\text{ and }\gamma_{k}=2\sin k\zeta
\sin\zeta\text{.}
\end{equation*}
Then, the blocks $B_{k}$ are%
\begin{equation*}
B_{k}=-\alpha_{k}I+\gamma_{k}(iJ)+2\mu^{2}h^{\prime}(\left\vert
\mu\right\vert ^{2})diag(1,0).
\end{equation*}

For the linearization of the equation one has that%
\begin{equation*}
P^{-1}M(\nu)P=diag(m_{1}(\nu),...,m_{n}(\nu))\text{.}
\end{equation*}
Thus, we find the matrices $m_{k}(\nu)$ in terms of the blocks $B_{k}$ as

\begin{equation*}
m_{k}(\nu)=-\nu(iJ)+B_{k}\text{ for }k\in\{1,...,n\}\text{.}
\end{equation*}

The action of $(\zeta,\zeta,\varphi)\in\mathbb{\tilde{Z}}_{n}\times S^{1}$
on $W_{k}$ is $\rho(\zeta,\zeta,\varphi)=e^{ik\zeta}e^{il\varphi}$.
Therefore, the isotropy group of the space $W_{k}$ is
\begin{equation*}
\mathbb{Z}_{n}(k)=\left\langle \left( \zeta,\zeta,-k\zeta\right)
\right\rangle .
\end{equation*}

\section{Bifurcation theorem}

The orthogonal degree is defined for orthogonal maps that are non-zero on
the boundary of some open bounded invariant set. The degree is made of
integers, one for each orbit type, and it has all the properties of the
usual Brouwer degree. Hence, if one of the integers is non-zero, then the
map has a zero corresponding to the orbit type of that integer. In addition,
the degree is invariant under orthogonal deformations that are non-zero on
the boundary. The degree has other properties such as sum, products and
suspensions. For instance, the degree of two pieces of the set is the sum of
the degrees. The interested reader may consult \cite{IzVi03}, \cite{BaKrSt06} and
\cite{Ry05} for more details on equivariant degree and degree for gradient maps.

Now, if one has an isolated orbit, then its linearization at one point of
the orbit $x_{0}$ has a block diagonal structure, due to Schur's lemma,
where the isotropy subgroup of $x_{0}$ acts as $\mathbb{Z}_{n}$ or as $S^{1}$%
. Therefore, the orthogonal index of the orbit is given by the signs of the
determinants of the submatrices where the action is as $\mathbb{Z}_{n}$, for
$n=1$ and $n=2$, and the Morse indices of the submatrices where the action
is as $S^{1}$. In particular, for problems with a parameter, if the
orthogonal index changes at some value of the parameter, one will have
bifurcation of solutions with the corresponding orbit type. Here, the
parameter is the frequency $\nu$.

The fixed point subspace of the isotropy group $\Gamma_{\bar{a}}\times S^{1}$
corresponds to the block $m_{n}(0)=B_{n}$. Since the generator of the kernel
is $A_{1}\bar{a}=T_{n}(-n^{1/2}e_{2})$, then $e_{2}$ must be in the kernel
of $m_{n}(0)$.

Following \cite{IzVi03}, one defines $\sigma$ to be the sign of $m_{n}(0)$
in the orthogonal subspace to $e_{2}$. Since $B_{n}=\mu^{2}h^{\prime}(\mu
^{2})diag(1,0)$, then%
\begin{equation*}
\sigma=sgn(e_{1}^{T}B_{n}e_{1})=sgn(h^{\prime}(\mu^{2})).
\end{equation*}
We have proved, in \cite{GaIz11}, that $m_{k}(0)=B_{k}$ is invertible except
for a point $\mu_{k}$ for $k=1,..,n-1$, solution of $\mu^{2}h^{\prime}(\mu
^{2})=\delta_{k}$, with%
\begin{equation*}
\delta_{k}=(\alpha_{k}^{2}-\gamma_{k}^{2})/(2\alpha_{k}).
\end{equation*}

\begin{definition}
Following \cite{IzVi03}, we define%
\begin{equation*}
\eta_{k}(\nu_{0})=\sigma\{n_{k}(\nu_{0}-\rho))-n_{k}(\nu_{0}+\rho)\}\text{,}
\end{equation*}
where $n_{k}(\nu)$ is the Morse index of $m_{k}(\nu)$.
\end{definition}

This number corresponds to the jump of the orthogonal index at $\nu_{0}$.
Then, from the results of \cite{IzVi03}, we can state the following theorem
for $\mu\neq\mu_{1},...,\mu_{n-1}$.

\begin{theorem}
If $\eta_{k}(\nu_{k})$ is different from zero, then the relative equilibrium
has a global bifurcation of periodic solutions from $2\pi/\nu_{k}$ with
isotropy group $\mathbb{\tilde{Z}}_{n}(k)$.
\end{theorem}

The solutions with isotropy group $\mathbb{\tilde{Z}}_{n}(k)$ must satisfy
the symmetries

\begin{equation*}
u_{j+1}(t)=e^{\ ij\zeta}u_{1}(t+jk\zeta).
\end{equation*}

The norms of these oscillators are related by the formula $%
r_{j+1}(t)=r_{1}(t+jk\zeta)$, in particular, for $k=n$, all oscillate in an
identical fashion. If $k$ divides $n$, then one will have $k$ equal
traveling waves, each one formed by $n/k$ oscillators.

By global bifurcation, we mean that the branch goes to infinity in norm or
period (for these possibilities we say that the bifurcation is non
admissible) or, if not, then the sum of the above jumps, over all the
bifurcation points, is zero.

A complete description of these solutions may be found in the paper \cite%
{GaIz112}.

\section{Spectral analysis}

The $m_{k}(\nu)$ have real eigenvalues, so the matrices $m_{n-k}(\nu)$ and
$m_{k}(-\nu)$ have the same spectrum due to the equality $m_{n-k}(\nu)=\bar
{m}_{k}(-\nu)$. As a consequence, the Morse numbers satisfy
\[
n_{n-k}(\nu)=n_{k}(-\nu).
\]
For $n=4$, one has that $\alpha_{k}=0$ and the determinant of $m_{k}(\nu)$ is
$-(\gamma_{k}-\nu)^{2}$, and there is no jump in the Morse number. Note also
that, for $k=n$, one has $\alpha_{n}=\gamma_{n}=0$ and the determinant is
$-\nu^{2}$, that is $m_{n}(\nu)$ is invertible for $\nu\neq0$ and there is no
bifurcation of periodic solutions with that symmetry.

Since we are going to treat the case $n=3$ separately, we shall suppose for
now that $n\geq5$. For these values of $n$ one has that $\alpha_{k}>0$ and
$\delta_{k}<\alpha_{k}/2$.

\begin{proposition}
Define $\nu_{\pm}$ as
\[
\nu_{\pm}=\gamma_{k}\pm\sqrt{\alpha_{k}\left(  \alpha_{k}-2\mu^{2}h^{\prime
}(\mu^{2})\right)  }\text{.}%
\]
For $n\geq5$, the matrix $m_{k}(\nu)$ changes Morse index only for $\nu_{\pm}
$, whenever $\mu^{2}h^{\prime}(\mu^{2})<\alpha_{k}/2$. In this case%
\[
\eta_{k}(\nu_{\pm})=\pm\sigma\text{,}%
\]
with $\sigma=sgn(h^{\prime}(\mu^{2}))$. Furthermore, the values $\nu_{\pm}$
are positive only in the following cases:

\begin{description}
\item[(a)] The value $\nu_{+}$ is positive for $k\in\{1,...,n-1\}$ and
$\mu^{2}h^{\prime}(\mu^{2})<\delta_{k}$.

\item[(b)] The values $\nu_{+}$ and $\nu_{-}$ are positive for $k\in
\{1,..,[n/2]\}$ and $\delta_{k}<\mu^{2}h^{\prime}(\mu^{2})<\alpha_{k}/2$.
\end{description}
\end{proposition}

\begin{proof}
Since the determinant of $m_{k}$ is%
\[
d_{k}(\nu)=\det m_{k}=-2\alpha_{k}\mu^{2}h^{\prime}(\mu^{2})+\alpha_{k}%
^{2}-(\gamma_{k}-\nu)^{2}\text{,}%
\]
then $d_{k}(\nu)$ is zero only for $\nu_{\pm}$, if $\mu^{2}h^{\prime}(\mu
^{2})<\alpha_{k}/2$. Furthermore, the trace of $m_{k}(\nu)$ is
\[
T_{k}=2\mu^{2}h^{\prime}(\left\vert \mu\right\vert ^{2})-2\alpha_{k}%
<-\alpha_{k}<0\text{.}%
\]
Since $d_{k}(\nu)$ is positive in $(\nu_{-},\nu_{+})$ and negative in the
complement, then $n_{k}(\nu)=2$ in $(\nu_{-},\nu_{+})$ and $n_{k}(\nu)=1$ in
the complement. Thus, $\eta_{k}(\nu_{-})=\sigma(1-2)$ and $\eta_{k}(\nu
_{+})=\sigma(2-1)$.

Since $d_{k}(\nu)$ is a polynomial of degree two, then $\nu_{+}$ is positive
and $\nu_{-}$ is negative, if $d_{k}(0)$ is positive. We conclude the result
for (a), that is $d_{k}(0)$ is positive for $\mu^{2}h^{\prime}(\mu^{2}%
)<\delta_{k}$. Furthermore, $d_{k}(0)$ is negative for $\delta_{k}<\mu
^{2}h^{\prime}(\mu^{2})$, thus, $\nu_{+}$ and $\nu_{-}$ have the same sign.
Since $\nu_{\pm}=\gamma_{k}\pm\sqrt{\ast}$, then $\nu_{\pm}$ have the sign of
$\gamma_{k}$. We conclude the result for (b), that is $\gamma_{k}$ is positive
for $k\in\{1,..,[n/2]\}$ and that $\gamma_{n-k}=-\gamma_{k}$.
\end{proof}

Thus, for the lattice with a general potential for $n\geq5$ we have:

\begin{theorem}
For each $k\in\{1,...,n-1\}$ such that $\mu^{2}h^{\prime}(\mu^{2})<\delta_{k}%
$, the relative equilibrium has a global bifurcation of periodic solutions
starting from the period $2\pi/\nu_{+}$ with symmetries $\mathbb{\tilde{Z}%
}_{n}(k)$ \emph{.} Furthermore, this bifurcation is non-admissible or goes to
another equilibrium.

For each $k\in\{1,...,[n/2]\}$ such that $\delta_{k}<\mu^{2}h^{\prime}(\mu
^{2})<\alpha_{k}/2$, the equilibrium has two global bifurcations of periodic
solutions starting from the periods $2\pi/\nu_{+}$ and $2\pi/\nu_{-}$ with
symmetries $\mathbb{\tilde{Z}}_{n}(k)$.
\end{theorem}

\begin{remark}
We have proven that the determinant of the block $m_{k}(\nu)$ is zero at two
values for $k\in\{1,...,n-1\}$ if $\mu^{2}h^{\prime}(\mu^{2})<\alpha_{k}/2$.
Thus, the block $m_{k}$ corresponds to stable solutions of the linear equation
for $\mu^{2}h^{\prime}(\mu^{2})<\alpha_{k}/2$. On the other hand, $\det
m_{n}(\nu)$ has a double zero $\nu=0$, but the block $m_{n}$ gives stable
solutions due to the spatial symmetries of the problem. Using the fact that
$\alpha_{k}$ are increasing for $k\in\{1,...,[n/2]\}$, we conclude that the
solution $\bar{a}$ is linearly stable for
\[
\mu^{2}h^{\prime}(\mu^{2})<\alpha_{1}/2.
\]

\end{remark}

The only remaining case is $n=3$. The proofs are similar to the previous case,
by taking the reverse inequalities.

\begin{proposition}
For $n=3$, the determinant, $\det m_{k}$, is zero only at $\nu_{\pm}$ for
$\alpha_{k}/2<\mu^{2}h^{\prime}(\mu^{2})$. In this case one has that $\eta
_{k}(\nu_{\pm})=\mp\sigma$ with $\sigma=sgn(h^{\prime}(\mu^{2}))$.
Furthermore, the values $\nu_{\pm}$ are positive only in the following cases:

\begin{description}
\item[(a)] If $\nu_{+}$ is positive for $k\in\{1,2\}$ and $0<\mu^{2}h^{\prime
}(\mu^{2})$.

\item[(b)] If $\nu_{+}$ and $\nu_{-}$ are positive for $\alpha_{1}/2<\mu
^{2}h^{\prime}(\mu^{2})<0$.
\end{description}
\end{proposition}

\begin{proof}
Since $\alpha_{k}<0$, then $\det m_{k}$ is zero at $\nu_{\pm}$ if $\alpha
_{k}/2<\mu^{2}h^{\prime}(\mu^{2})$. In this case the trace of $m_{k} $ is
\[
T_{k}(\lambda)=2\mu^{2}h^{\prime}(\left\vert \mu\right\vert ^{2})-2\alpha
_{k}>-\alpha_{k}>0\text{.}%
\]
Hence, $n_{k}(\nu)=0$ at $(\nu_{-},\nu_{+})$ and $n_{1}(\nu)=1$ on the
complement. Thus, $\eta_{k}(\nu_{-})=\sigma(1-0)$ and $\eta_{k}(\nu
_{+})=\sigma(0-1)$.

Furthermore, $\nu_{+}$ is positive when $d_{k}(0)$ is positive. Since
$\alpha_{1}=-\gamma_{1}$, then $\delta_{1}=\delta_{2}=0$. We conclude the
result for (a) from the fact that $d_{k}(0)$ is positive for $0=\delta_{k}%
<\mu^{2}h^{\prime}(\mu^{2})$ . Also, $d_{k}(0)$ is negative for $\mu
^{2}h^{\prime}(\mu^{2})<\delta_{k}=0$, then $\nu_{+}$ and $\nu_{-}$ have the
same sign. Since $\nu_{\pm}=\gamma_{k}\pm\sqrt{\ast}$, then $\nu_{\pm}$ have
the sign of $\gamma_{k}$. We get the result for (b), from $\gamma_{1}%
=2\sin^{2}\zeta>0$ and that $\gamma_{2}=-\gamma_{1}$.
\end{proof}

\emph{Hence, for the lattice and} $n=3$\emph{, the relative equilibrium has,
for each} $k=1,2$, \emph{if} $0<\mu^{2}h^{\prime}(\mu^{2})$\emph{, a global
bifurcation of periodic solutions starting from the period} $2\pi/\nu_{+}$
\emph{with symmetries} $\mathbb{\tilde{Z}}_{3}(k)$\emph{.} Furthermore, this
bifurcation is inadmissible or goes to another equilibrium.

\emph{For }$n=3$\emph{\ the relative equilibrium, for }$\alpha_{1}/2<\mu
^{2}h^{\prime}(\mu^{2})<0$\emph{,\ has global bifurcations of periodic
solutions starting from the periods }$2\pi/\nu_{+}$\emph{\ and }$2\pi/\nu_{-}%
$\emph{\ with symmetries} $\mathbb{\tilde{Z}}_{3}(1)$\emph{.}

For $n=3$, the relative equilibrium is linearly stable for $\alpha_{1}%
/2<\mu^{2}h^{\prime}(\mu^{2})$.

\subsubsection{Schr\"{o}dinger potential}

For the lattice for the cubic Schr\"{o}dinger potential, we have $h(\mu)=\mu$
with%
\[
h^{\prime}(\mu^{2})=1\text{ and }\sigma=1.
\]
For $n\geq5$, one has the following cases:

For $k\in\{1,2,n-2,n-1\}$ we have $\delta_{k}\leq0$, then condition (a) is
never satisfied and, for $\mu\in(0,\sqrt{\alpha_{k}/2})$, condition (b) holds.

For $k\in\{3,...,n-3\}$ we have $\delta_{k}>0$, then, for $\mu\in
(0,\sqrt{\delta_{k}})$, condition (a) is satisfied, while, for $\mu\in
(\sqrt{\delta_{k}},\sqrt{\alpha_{k}/2})$, condition (b) holds.

\begin{theorem}
Thus, for the lattice with a Schr\"{o}dinger potential, for $n\geq6$, the
relative equilibrium has, for each  $k\in\{3,...,n-3\}$ and $\mu\in
(0,\sqrt{\delta_{k}})$, a global bifurcation of periodic solutions, starting
from the period $2\pi/\nu_{+}$ with symmetries $\mathbb{\tilde{Z}}_{n}(k)$.
Furthermore, this bifurcation is inadmissible or goes to another equilibrium.

For $n\geq5$\ the relative equilibrium has, for each $k\in\{1,2\}$ such that
$\mu\in(0,\sqrt{\alpha_{k}/2})$ and, for each $k\in\{3,...,[n/2]\}$ such that
$\mu\in(\sqrt{\delta_{k}},\sqrt{\alpha_{k}/2})$, two global bifurcations of
periodic solutions starting from the periods $2\pi/\nu_{+}$ and $2\pi/\nu_{-}$
with symmetries $\mathbb{\tilde{Z}}_{n}(k)$.
\end{theorem}

For $n\geq5$, the relative equilibrium is linearly stable for the amplitudes
$\mu\in(0,\sqrt{\alpha_{1}/2})$.

For $n=3$ and $k\in\{1,2\}$, we have that condition (a) is always satisfied.
\emph{For} $n=3$,\emph{\ the relative equilibrium has, for each} $k\in\{1,2\}
$ \emph{such that} $\mu\in(0,\infty)$\emph{, a global bifurcation of periodic
solutions starting from the period} $2\pi/\nu_{+}$ \emph{with symmetries}
$\mathbb{\tilde{Z}}_{3}(k)$\emph{.} Furthermore, for $n=3$, this relative
equilibrium is linearly stable for the amplitudes $\mu\in(0,\infty)$.

\subsubsection{Saturable potential}

For the lattice with the saturable potential, one has $h(x)=(1+x)^{-1}$, with%
\[
h^{\prime}(\mu^{2})=-(1+\mu^{2})^{-2}\text{ and }\sigma=-1\text{.}%
\]
$\ $For $n\geq5$ we have the following cases:

For $k\in\{2,...,n-2\}$ one has $\delta_{k}\geq0$, then, for each $\mu
\in(0,\infty)$ one gets $\mu^{2}h^{\prime}(\mu^{2})<\delta_{k}$, that is
condition (a).

For $n\in\{16,17,...\}$, let $\mu_{-}\in(0,1)$ and $\mu_{+}\in(1,\infty)$ be
the solutions of $\mu^{2}h^{\prime}(\mu^{2})=\delta_{1}$. Then, for $\mu
\in(0,\mu_{-})\cup(\mu_{+},\infty)$, one has $\delta_{1}<\mu^{2}h^{\prime}%
(\mu^{2})$, that is condition (b) and, for $\mu\in(\mu_{-},\mu_{+})$ one has
$\mu^{2}h^{\prime}(\mu^{2})<\delta_{1}$, that is condition (a). Furthermore,
for $n=\{5,...,15\}$ one has the same result as before provided one defines
$\mu_{-}$ and $\mu_{+}$ to be zero.

\begin{theorem}
Thus, for the lattice with the saturable potential, for $n\geq5$, the relative
equilibrium has, for each $k\in\{2,...,n-2\}\ $and $\mu\in(0,\infty)$ and, for
each $k\in\{1,n-1\}\ $and $\mu\in(\mu_{-},\mu_{+})$, a global bifurcation of
periodic solutions starting from the period $2\pi/\nu_{+}$ with symmetries
$\mathbb{\tilde{Z}}_{n}(k)$. This bifurcation is inadmissible or goes to
another equilibrium.

Moreover, the relative equilibrium has, for each $\mu\in(0,\mu_{-})\cup
(\mu_{+},\infty)$, two global bifurcations of periodic solutions starting from
the periods $2\pi/\nu_{+}$ and $2\pi/\nu_{-}$ with symmetries $\mathbb{\tilde
{Z}}_{n}(1)$.
\end{theorem}

For $n\geq5$, the relative equilibrium is linearly stable for all amplitudes
$\mu\in(0,\infty)$.

For $n=3$, we have that $\alpha_{1}/2=-3/4$, hence, $\alpha_{1}/2<-1/4<\mu
^{2}h^{\prime}(\mu^{2})$ for all $\mu\in(0,\infty)$. Thus, condition (b) is
always satisfied. \emph{For} $n=3$,\emph{\ the relative equilibrium has, for
each } $\mu\in(0,\infty)$\emph{, global bifurcations of periodic solutions
starting from the periods} $2\pi/\nu_{+}$ \emph{and} $2\pi/\nu_{-}$ \emph{with
symmetries} $\mathbb{\tilde{Z}}_{3}(1)$\emph{. } Furthermore, for $n=3$, the
relative equilibrium is linearly stable for the amplitudes $\mu\in(0,\infty)$.


\end{document}